\documentclass[12pt]{article}

\usepackage{amssymb}
\usepackage{amsmath}
\usepackage{graphicx}
\usepackage{enumerate}
\usepackage{mathrsfs}
\usepackage{latexsym}
\usepackage{psfrag}
\usepackage{mathrsfs}

\usepackage{tikz} 
\usepackage{mathtools}
\usepackage{setspace} 

\newcommand{\qed}{\hfill $\square$\\
\vspace{0.1cm}}

\newcommand{\B}{\mathcal{B}}

\newcommand{\F}{\mathcal{F}}

\newcommand{\I}{\mathcal I}

\newcommand{\Pa}{\mathcal{P}}

\newcommand{\V}{\mathcal{V}}

\newcommand{\floor}[1]{\lfloor{ #1}\rfloor}

\newcommand{\pa}{\operatorname{Pa}}
\newcommand{\pas}{\operatorname{Pa}^*}
\newcommand{\abs}[1]{\left| #1 \right|}

\newtheorem{theorem}{Theorem}[section]
\newtheorem{lemma}[theorem]{Lemma}
\newtheorem{proposition}[theorem]{Proposition}

\newtheorem{conjecture}{Conjecture}[section]

\newenvironment{proof}{\noindent{\em Proof.}}{\qed}

\newcommand{\nchn}{\dbinom{n}{\lfloor \frac{n}{2}\rfloor}}

\newcommand{\comments}[1]{}

\begin{document}
\title{Packing Posets in the Boolean Lattice}

\author{Andrew P. Dove
\thanks{Department of Mathematics, University of South Carolina,
Columbia, SC, USA 29208 (doveap@mailbox.sc.edu).}
\and Jerrold R. Griggs
\thanks{Department of Mathematics, University of South Carolina,
Columbia, SC, USA 29208 (griggs@math.sc.edu).
Research supported in part by a grant from the Simons Foundation (\#282896 to Jerrold Griggs).} }
\date{September 15, 2013\\}

\maketitle

\begin{abstract}
We are interested in maximizing the number of pairwise unrelated copies of a poset $P$ in the family of all subsets of $[n]$.
For instance, Sperner showed that when $P$ is one element, $\nchn$ is the maximum number of copies of $P$.
Griggs, Stahl, and Trotter have shown that when $P$ is a chain on $k$ elements, $\dfrac{1}{2^{k-1}}\nchn$ is asymptotically the maximum number of copies of $P$.
We prove that for any $P$ the maximum number of unrelated copies of $P$ is asymptotic to a constant times $\nchn$.
Moreover, the constant has the form $\dfrac{1}{c(P)}$, where $c(P)$ is the size of the smallest convex closure over all embeddings of $P$ into the Boolean lattice.
\end{abstract}



\section{Introduction}\label{sec:Intro}

Using standard notation, let $\B_n$ be the inclusion poset of all subsets of $[n]=\{1,2,\dots,n\}$. Let $P$ be any poset.
Let $f:P\to\B_n$ be a weak embedding of the poset $P$ into $\B_n$, i.e., if $a<b\in P$, then $f(a)\subset f(b)$.
We call $f(P)$ a copy of $P$ in $\B_n$. Let $\{\F_i\}$ be pairwise unrelated copies of $P$, i.e., if $A_i\in\F_i$, $A_j\in\F_j$, and $i\ne j$, then $A_i$ and $A_j$ are unrelated.
We say the family $\F=\cup_i\F_i$ is a family constructed from pairwise unrelated copies of $P$.
Let $\pa(n,P)$ denote the maximum size of a family constructed from pairwise unrelated copies of $P$ in $\B_n$.
This quantity can be generalized to apply to a collection of posets; let $\pa(n,\{P_i\})$ denote the maximum size of a family in $\B_n$ constructed from pairwise unrelated copies of posets chosen from the collection of posets $\{P_i\}$.

The motivation for finding $\pa(n,\{P_i\})$ comes from a question under intensive study in recent years, that of finding the maximum size $\operatorname{La}(n,Q)$, which is the maximum size of a family $\F\subseteq\B_n$ that contains no copy of poset $Q$.
This seems to be a challenging problem in extremal set theory, even determining the asymptotic growth of $\operatorname{La}(n,Q)$, as $n\rightarrow \infty$, for posets as simple as the four element diamond (which is $\B_2$).
For a survey on the topic, see \cite{GriLiLu}.
For the most recent progress on the diamond, see \cite{KraMarYou}.

It is natural to extend this notion to collections of posets $\{Q_j\}$, seeking to find the maximum size $\operatorname{La}(n,\{Q_j\})$ of a family $\F\in\B_n$ that contains no copy of any poset $Q_j$ in the collection.
We noticed that for the collection $\{\V,\Lambda\}$, where $\V=\V_2$ is the poset on $\{a,b,c\}$ with $a<b$ and $a<c$, and $\Lambda$ is the poset on $\{a,b,c\}$ with $a>b$ and $a>c$, $\operatorname{La}(n,\{\V,\Lambda\})$ is the same as $\pa(n,\{\B_0,\B_1\})$, since any collection of subsets that contains no copy of $\V$ or $\Lambda$ has components consisting only of single sets and/or two-element chains, all unrelated to each other.
We recently learned that Katona and Tarj\'an \cite{KatTar} solved this very problem years ago, showing that $\operatorname{La}(n,\{\V,\Lambda\})=2\dbinom{n-1}{\lfloor\frac{n-1}{2}\rfloor}$.
We were able to derive the same result, applying a 1984 result of Griggs, Stahl, and Trotter \cite{GriStaTro} that gives $\pa(n,\B_1)$; we present their more general result for the path $\Pa_k$ below.

More generally, for any collection $\{Q_j\}$, $\operatorname{La}(n,\{Q_j\})$ is equivalent to $\pa(n,\{P_i\})$, where $\{P_i\}$ is the collection of all possible connected posets that do not contain any of the posets in $\{Q_j\}$ as a subposet.
Note that the collection $\{P_i\}$ may be infinite.
For instance, $\operatorname{La}(n,\V)$ is the same as $\pa(n,\{P_i\})$ where $P_i$ is the $i$-fork consisting of one set that contains $i$ (unrelated) sets, $i\ge0$.
So the problem of determining $\pa(n,\{P_i\})$ can be viewed as more general than the $\operatorname{La}(n,\{Q_j\})$ problem.

In this paper we now concentrate on finding the asymptotic behavior of $\pa(n,P)$ for single posets $P$.
We hope that these ideas might help with solving the more difficult problem of finding $\operatorname{La}(n,\{Q_j\})$.
To start us off, here are some examples in the literature of finding $\pa(n,P)$ for specific $P$.
A natural technique in proving results on posets is to count full chains.
We define a \emph{full chain} in $\B_n$ to be a chain of $\B_n$ that includes a subset of $[n]$ of every size.
In Sperner's classic theorem (1928), he finds $\pa(n,P)$ for $P=\B_0$.

\begin{theorem}[\cite{Spe}]
$\pa(n,\B_0)$ is $\nchn$.
\end{theorem}

We include a proof of Sperner's Theorem to demonstrate a basic outline for our later proofs.
Here is a proof introduced by Lubell \cite{Lub}:

\begin{proof}
Each copy of $\B_0$ is just a subset of $[n]$.
If a copy of $\B_0$ is a subset of size $a$, then it meets $a!(n-a)!$ full chains in $\B_n$;
this is at least $\lfloor n/2\rfloor!\lceil n/2\rceil!$ full chains.
There are only $n!$ full chains in $\B_n$. No chain may hit more than one copy of $\B_0$.
The number of copies of $\B_0$ is $\dfrac{\pa(n,\B_0)}{\abs{\B_0}}$.
Counting the full chains gives the inequality
$$\frac{\pa(n,\B_0)}{\abs{\B_0}}\Big(\lfloor n/2\rfloor!\lceil n/2\rceil!\Big)\leq n!,\ \text{and hence}$$
$$\pa(n,\B_0)\leq\abs{\B_0}\dfrac{n!}{\lfloor n/2\rfloor!\lceil n/2\rceil!}=\abs{\B_0}\nchn=\nchn.$$
This bound is tight, as the copies of $\B_0$ may be chosen to be the middle level of $\B_n$.
\end{proof}

This proof was generalized by Griggs, Stahl, and Trotter (1984) for $P=\Pa_k$, the chain (or path) on $k+1$ elements.  The earlier paper of Bollob\'as~\cite{Bol} also implies this result.

\begin{theorem}[\cite{GriStaTro}]
\label{chain:theorem}
$\pa(n,\Pa_k)$ is $(k+1)\dbinom{n-k}{\lfloor\frac{n-k}{2}\rfloor}\sim\dfrac{k+1}{2^{k}}\nchn$.
\end{theorem}

\begin{proof}
For a chain $\Pa$, with a minimum set $A\subseteq[n]$ and maximum $B\subseteq [n]$, define $\I_\Pa$ as $[A,B]$, the interval from $A$ to $B$.
For two chains $\Pa$ and $\Pa^\prime$ to be pairwise unrelated copies of $\Pa_k$, a full chain in $\B_n$ hits at most one of $\I_{\Pa}$ or $\I_{\Pa^\prime}$.
A full chain that hits $\I_\Pa$ is a chain constructed from all the elements of $A$ before any of the elements of $[n]\setminus B$.
There are $n-\abs{B\setminus A}$ elements in $A$ or not in $B$.
Therefore, the number of chains that hit $\I_\Pa$ is
$$\dfrac{n!}{\dbinom{n-\abs{B\setminus A}}{\abs A}}\geq \dfrac{n!}{\dbinom{n-k}{\lfloor\frac{n-k}{2}\rfloor}}$$
so each of the $\dfrac{\pa(n,\Pa_k)}{\abs{\Pa_k}}$ intervals from the copies of $\Pa_k$ meets at least $\dfrac{n!}{\dbinom{n-k}{\lfloor\frac{n-k}{2}\rfloor}}$ full chains.
This gives that
$$\frac{\pa(n,\Pa_k)}{\abs{\Pa_k}}\dfrac{n!}{\dbinom{n-k}{\lfloor\frac{n-k}{2}\rfloor}}\leq n!,\ \text{or}$$
$$\pa(n,\Pa_k)\leq\abs{\Pa_k}\dbinom{n-k}{\lfloor\frac{n-k}{2}\rfloor}=(k+1)\dbinom{n-k}{\lfloor\frac{n-k}{2}\rfloor}\sim\dfrac{k+1}{2^{k}}\nchn.$$
The bound is tight, as the following construction demonstrates.
Fix a set $S\subseteq[n]$ such that $\abs S=k$.
The $A$'s corresponding to the chains are all the $\lfloor\frac{n-k}{2}\rfloor$-sets of $[n]\setminus S$, and the $\Pa_k$'s are chosen as  any full chain in the interval $[A,A\cup S]$ for each $A$.
\end{proof}

Notice the steps in the proof above.
For the upper bound, first, each copy of $\Pa_k$ is contained in a larger set system $\I_\Pa$, where a full chain hits at most one $\I_\Pa$.
Second, a lower bound on the number of full chains that hit an $\I_\Pa$ is found.
Now $n!$ divided by this lower bound is an upper bound on $\dfrac{\pa(n,P)}{\abs{P}}$, the number of copies of $P$.
As for the lower bound, a construction is found.
In this particular example, the construction is multiple copies of $\Pa_k$, where each copy's minimum is on a base rank $\lfloor\frac{n-k}{2}\rfloor$, and the minimum is constructed from just the elements of $[n]\setminus S$.
The set $S$ restricts the choices of which sets in the base level to include in the packing.
Each copy of $P$ can then easily be built on top of its minimum using the elements of $S$.

\comments{
\begin{lemma}
\label{asymptotic:lemma}
For constants $a$ and $b$, $\binom{n-a}{\lfloor n/2\rfloor-b}\sim\frac{1}{2^a}\nchn$.
\end{lemma}

\begin{proof}
\begin{align*}
\binom{n-a}{\lfloor n/2\rfloor-b}&=\bigg(\frac{(n-a)!}{(\lfloor n/2\rfloor-b)!(\lceil n/2\rceil-(a-b))!}\bigg)\bigg(\frac{(\lfloor n/2\rfloor)!(\lceil n/2\rceil)!}{n!}\bigg)\nchn
\\&=\bigg(\frac{(n-a)!}{n!}\bigg)\bigg(\frac{(\lfloor n/2\rfloor)!}{(\lfloor n/2\rfloor-b)!}\bigg)\bigg(\frac{(\lceil n/2\rceil)!}{(\lceil n/2\rceil-(a-b))!}\bigg)\nchn
\\&\sim n^{-a}(n/2)^b(n/2)^{a-b}\nchn
\\&=2^{-a}\nchn.
\end{align*}
\end{proof}
}

A similar method may be used to find $\pa(n,P)$ for general $P$.
One important concept is how each copy of $\Pa_k$ is contained in a larger set system $\I_\Pa$, where a full chain hits at most one $\I_\Pa$.
Similarly, we contain a copy of $P$ inside the convex closure of that copy.
The convex closure of a set system is defined as follows:
Let $\F\subseteq \B_n$. In $\B_n$, $\F$ generates an ideal (or down-set) and a filter (or up-set) denoted as follows:
$$D(\F)=\{S\in\B_n|S\subseteq A\ \text{for some}\ A\in \F\},\ \text{and}$$
$$U(\F)=\{S\in\B_n|A\subseteq S\ \text{for some}\ A\in \F\}.$$
We define a closure operator on $\F$ as $\overline{\F}:=D(\F)\cap U(\F)$. Another definition would be
$$\overline \F:=\{S\in\B_n| A\subseteq S\subseteq B\ \text{for some } A,B\in \F \}.$$
Here $S$, $A$, and $B$ could be equal, so clearly $\F\subseteq\overline \F$.
A family $\F$ such that $\F=\overline\F$ is called convex.
Note that convex families appear in the literature, including the conjecture by P. Frankl and J. Akiyama:
\begin{conjecture}
[\cite{AkiFra}]
For every convex family $\F\subseteq \B_n$, there exists an antichain $\mathcal A\subseteq\F$ such that $\abs{\mathcal A}/\abs\F\geq\nchn/2^n.$
\end{conjecture}

\comments{
\begin{lemma}\label{unrelated:closures:lemma}
If $\F_1,\F_2\subseteq\B_n$ are unrelated (any $A_1\in \F_1$ and $A_2\in \F_2$ are unrelated), then $\overline{\F_1}$ and $\overline{\F_2}$ are unrelated.
\end{lemma}

\begin{proof}
Assume by contradiction that there are $S_1\in \overline{\F_1}$ and $S_2\in \overline{\F_2}$ such that $S_1\subseteq S_2$. Then there exists by the definition of closure an $A_1\in \F_1$ and $B_2\in \F_2$ such that $A_1\subseteq S_1\subseteq S_2\subseteq B_2$, but then $A_1\subseteq B_2$.
\end{proof}
}

If two copies of $P$ are unrelated, then their closures must be unrelated as well.
Therefore, we are more interested in the size and structure of the closure of a copy of $P$ than of the copy of $P$ itself.
For a weak embedding $f$ of $P$ into $\B_k$, there exists a minimum value of $\abs{\overline{f(P)}}$ over all choices of $f$ and $k$. Denote this minimum as $c(P)$.

Here are some examples. If $P$ is the three-element poset $\V$, we have
that $f$ may be embed $\V$ into $\B_2$ such that $f(\V)=\{\emptyset,\{1\},\{2\}\}$.
Now $\overline{f(\V)}=f(\V)$ so $c(\V)=\abs{\V}=3$.
In the proof of Theorem \ref{chain:theorem}, the closure of an embedding of a chain $\Pa_k$ is the smallest interval in which it is enclosed, $\I_\Pa$ in the proof. The smallest size of this interval is $2^k$ so $c(\Pa_k)=2^k$.


Here is one of the two main theorems, finding $\pa(n,P)$ asymptotically for any $P$ in terms of $c(P)$ and $\abs P$.
\begin{theorem}
\label{main:one}
For any poset $P$, as $n\to\infty$, $\pa(n,P)\sim\dfrac{\abs P}{c(P)}\nchn$.
\end{theorem}

We may also ask the similar question, what is the maximum number of pairwise unrelated induced copies of $P$ in $\B_n$, where each copy is a strong embedding of $P$?
A strong embedding $f$ of $P$ is such that for $a,b\in P$, $a<b$ \emph{if and only if} $f(a)\subset f(b)$.
We will denote the maximum size of a family in $\B_n$ constructed from induced copies of $P$ as $\pas(n,P)$.
We can also define the more general quantity $\pas(n,\{P_i\})$.

We can similarly define $c^*(P)$ as the minimum size of the closure of a strong embedding of $P$ in $\B_n$ over all possible $n$.
In general, $c^*(P)\ne c(P)$. Take for instance the poset $J=\{a,b,c,d\}$, $a<b<c$, and $a<d$; $J$ may be weakly embedded into $\B_2$ so $c(J)=4$.
As for $f$, a strong embedding of $J$ into $\B_k$, there exists a set $B^\prime\in\B_k$, $f(b)\ne B^\prime$, such that $f(a)\subset B^\prime\subset f(c)$ so $B^\prime\in\overline{f(J)}$, but $f(d)\ne B^\prime$ because $f(d)\nsubseteq f(c)$.
Therefore, $c^*(J)\geq 5$.
Also, a strong embedding of $J$ into $\B_3$ is easy to find such that $c^*(J)=5$.

The second main theorem finds $\pa^*(n,P)$ asymptotically for any $P$ in terms of $c^*(P)$ and $\abs P$.
\begin{theorem}
\label{main:two}
For any poset $P$, as $n\to\infty$, $\pas(n,P)\sim\dfrac{\abs P}{c^*(P)}\nchn$.
\end{theorem}

While preparing this manuscript, we learned that this problem of determining asymptotically
the maximum number of unrelated copies of a poset $P$ in $\B_n$ was already proposed by
Katona at a conference lecture in 2010~\cite{Kat10}.  We also learned that
Katona and Nagy~\cite{KatNag} have recently (and independently) obtained
results essentially equivalent to our two main results above.  
Our extension of the problem to a family of posets, $\pa(n,\{P_i\})$, appears to be new.

The following two sections are a proof of Theorem~\ref{main:one}.
The proof of Theorem~\ref{main:two} will require only a few alterations.
This will be demonstrated after the main proof.


\section{The Upper Bound}\label{sec:UB}

We obtain the upper bound on the number of unrelated copies of
poset $P$ from an asymptotic lower bound on the number of full
chains that meet the \emph{closure} of a copy of $P$.  For a family $\F$ of subsets of
$[n]$, let $a(\F)$ be the number of full chains in $\B_n$ that
intersect $\F$.   While $a(\F)$ will be as large as $n!$, if, say,
$\F$ contains $\emptyset$, we are interested in how small it can
get.  If $\F$ consists of $m$ subsets of size $k$, then $a(\F)$
will be $m k! (n-k)!=m(n!/{\binom{n}{k}})$, which is at least
$m\Big(n!\big/\nchn\Big)$.  For fixed $m$, as $n$ grows we expect this last
formula to be the minimum asymptotically.  Let us denote by
$\overline a(m,n)$ the minimum of $a(\F)$, over all families $\F\subseteq
\B_n$ with $|\F|=m$.

\begin{proposition}\label{thm:amn}
Let integer $m\ge1$.  Then as $n\rightarrow \infty$ the minimum
number of full chains in $\B_n$ that meet a family of $m$ subsets
in $\B_n$, $\overline a(m,n)\sim m\Big(n!\big/\nchn\Big)$.
\end{proposition}

\begin{proof}
Let $\F=\{A_1,\ldots,A_m\}$ be a family of $m$ subsets of $[n]$.
For convenience let us assume that the subsets are labeled so
that for all $i<j$, $|A_i|\le |A_j|$.   For any $1\le
i_1<\cdots,i_k\le m$ let $b(i_1,\ldots,i_k)$ denote the number of
full chains that pass through all of $A_{i_1},\ldots,A_{i_k}$. Of
course, $b(i_1,\ldots,i_k)$ is nonzero if and only if the sets
$A_{i_1},\ldots,A_{i_k}$ form a chain.
Inclusion-exclusion gives us that $a(\F)$ is the sum of the
$b(i_1)$ minus the sum of the $b(i_1,i_2)$ plus the sum of the
$b(i_1,i_2,i_3)$ minus and so on. Our difficulty now is that some
terms $b(i_1,\ldots,i_k)$ with $k\ge2$ can actually be large
compared to some singleton terms $b(i_1)$, so we cannot
immediately dismiss them.  For instance, if $n=100$ and $\F$
happens to be a chain with $|A_i|=i$ for all $i$, then
$b(1,2)=1!1!98!$ is much larger than $b(50)=50!50!$. However, we
can exploit the fact that terms $b(i_1,\ldots,i_k)$ with $k\ge2$
are considerably smaller than some $b(i_1)$ terms. In the example,
we could instead compare $b(1,2)$ to $b(1)=1!99!$.

By making all signs for terms with $k\ge2$ negative, our
alternating sum lower bound above is at least the sum of the
$b(i_1)$ minus the sum over all $k\ge2$ of the terms
$b(i_1,\ldots,i_k)$. For the $2^m-m-1$ terms being subtracted, we
assign each one to a particular positive singleton term $b(j)$ as
follows:  For a term $b(i_1,\ldots,i_k)$ with $k\ge2$, by our
labeling we have $|A_{i_1}|\le\cdots\le |A_{i_k}|$.  Let
$u:=|A_{i_1}|$ and $v:=|A_{i_k}|$.  We assign this term to one of
$b(i_1)$ or $b(i_k)$, resp., according to whether $|u-(n/2)|$ is
at least (less than, resp.) $|v-(n/2)|$. For instance in the
example above, the terms $b(20,28)$ and $b(20,30,80)$ are assigned
to $b(20)$, while $b(20,30,81)$ is assigned to $b(81)$.

We have then each singleton term $b(j)=|A_j|!(n-|A_j|)!$.  There
are less than $2^{m-1}$ terms $b(i_1,\ldots,i_k)$ with $k\ge2$
assigned to $b(j)$.  For those terms that are nonzero, it means
that $A_{i_1}\subset \cdots\subset A_{i_k}$ and either $i_1$ or
$i_k$ is $j$, according to which is farther from $n/2$.  Suppose
$j=i_1$ (so $i_1<n/2$).  Then this term $b(i_1,\ldots,i_k)$ is a
product of factorials that refines $b(i_1)$:  While $i_1!$ is
still a factor, $(n-i_1)!$ is replaced by a product of factorials
no more than $1!(n-i_1-1)$, so in total, we get at most $b(i_1)$
divided by $(n-i_1)$, which is at least $n/2$.  In this case, and
similarly when $j=i_k$, we see that the term $b(i_1,\ldots,i_k)$
is at most $b(j)$ divided by $n/2$.  Therefore, the sum of all the
terms assigned to $b(j)$ is at most $b(j)$ times $2^m/n$.  Hence,
$$a(\F)\ge \sum_{j=1}^m b(j) (1-(2^m/n)). $$ Since each term
$b(j)=j!(n-j)!\ge n!\big/\nchn$, and this bound holds independent of
$\F$, we see that as $n\rightarrow\infty$ for fixed $m$,
$\overline a(m,n)\sim m\Big(n!\big/\nchn\Big)$.
\end{proof}

Now we consider our poset packing problem.  Assume that we have
$\pa(n,P)/\abs P$ unrelated copies $F_i$ of our poset $P$ contained in the
Boolean lattice $\B_n$.  In fact, if a full chain passes through
the closure $\overline{F_i}$ of one of these families
$F_i$, it does not pass through the closure of any other $F_j$,
since $F_i$ and $F_j$ are unrelated.   That is, the closures
$\overline{F_i}$ are also unrelated.  Each closure
$\overline{F_i}$ has at least $m=c(P)$ subsets in it so it meets at
least $\overline a(m,n)$ full chains.

Altogether, the number of full chains that meet some closure
$\overline{F_i}$ is then at least $\overline a(m,n)\pa(n,P)/\abs P$. This is
in turn at most the total number of full chains, $n!$.  Hence,
$\pa(n,P)/\abs P$ is at most $n!/\overline a(m,n)$, which is asymptotic to
$(1/m)\nchn$ for large $n$. This gives the desired asymptotic upper bound.

\section{The Lower Bound Construction}\label{sec:LB}

Let $m$, $k$, and $f$ be such that $f$ embeds $P$ into $\B_k$, and $\abs{\overline{f(P)}}=m=c(P)$.
We will construct an $\F\subseteq\B_n$ from pairwise unrelated copies of $f(P)$ so that the number of copies of $P$ in $\F$ is $\dfrac{\abs\F}{\abs P}\sim\dfrac{1}{m}\nchn$.

We will construct $F$ through a finite number of iterations.
Fix an $i\in\mathbb N$.
This $i$ is the number of iterations for which we construct asymptotically $\dfrac{(2^k-m)^{j}}{(2^k)^{j+1}}\nchn$ unrelated copies of $P$ for each $0\leq j\leq i-1$.
Because we may choose $i$ to be arbitrarily large, we will have
\begin{align*}
\frac{\abs\F}{\abs P}&\sim\sum_{j=0}^{i-1}\bigg(\frac{(2^k-m)^{j}}{(2^k)^{j+1}}\bigg)\nchn\\
&\sim\sum_{j=0}^{\infty}\bigg(\frac{(2^k-m)^{j}}{(2^k)^{j+1}}\bigg)\nchn\\
&=\frac{1}{2^k}\bigg[\frac{1}{1-\frac{2^k-m}{2^k}}\bigg]\nchn=\frac{1}{m}\nchn.
\end{align*}

Let's now create such an $\F\subseteq\B_n$ for each $n$.
For the rest of the argument, let $(A+x)$ be the translation $\{a+x\mid a\in A\}$ for a set $A\subseteq[n]$ and an integer $x$.
For the ease of notation, define $S_j:=[k(j+1)]\setminus[kj]=\{kj+1,kj+2,\dots,kj+k\}=([k]+kj)$, the set $[k]$ translated by a multiple of $k$.

A level (or row) of $\B_n$ is all subsets of $[n]$ of the same size, the rank of the level.
The level of rank $r$ is often denoted as $\binom{[n]}{r}$.
Define a \emph{layer} of $\B_n$ (denoted as $\ell$) to be $k+1$ consecutive levels of $\B_n$.
We call the smallest rank in layer $\ell$ its base rank, $b_\ell$.
Specifically, $\ell=\binom{[n]}{b_\ell}\cup\binom{[n]}{b_\ell+1}\cup\dots\cup\binom{[n]}{b_\ell+k}$.
We define our layers by taking the base ranks to be $\floor{n/2}+z(k+1)$ for all integers $z$; in this way, we partition the levels of $\B_n$ and any two layers are disjoint.
We construct $\F$ by populating certain layers with many copies of $f(P)$.
A layer $\ell$ that is populated corresponds to a triple $(j_\ell,R_\ell,b_\ell)$; $\ell$ has base rank $b_\ell$, the iteration in which it is populated $j_\ell$ (ranges from 0 to $i-1$), and a restriction set $R_\ell\subseteq [kj_\ell]$, which defines which elements of $\ell$ are in $\F$.
The following is exactly how $\F$ is constructed in a layer $\ell$:
\begin{align*}
\ell\cap \F=\big\{R_\ell\cup A\cup B\ \big|
&A\subseteq S_{j_\ell},(A-kj_\ell)\in f(P),\\
&B\subseteq[n]\setminus[k(j_\ell+1)],\ \text{and}\\
&\abs{R_\ell}+\abs{B}=b_\ell\big\}.
\end{align*}
Our choice for the $R_\ell$ and the order of the $b_\ell$'s, as we will show later, prevents any two copies of $P$ in different layers from having any related sets.
For a fixed $B$, the family of all the $A$'s forms a copy of $P$ that is $f(P)$ translated, from using the elements in $[k]$ to using the elements from $S_{j_\ell}$.
There is then one copy of $P$ in $\ell$ for each choice of $B$.
The purpose of $B$ is to combine with $R_\ell$ to be in the base level of the layer, i.e., $B\cup R_\ell\in\binom{[n]}{b_\ell}$.
There are $\dbinom{n-k(j_\ell+1)}{b_\ell-\abs{R_\ell}}$ choices for $B$.
Notice that copies of $P$ within a layer are unrelated; every set in a copy of $P$ has the same base set $R_\ell\cup B$, and the copies of $P$ in a layer have unrelated base sets.

For each iteration $j$, we will be populating $(2^k-m)^{j}$ layers.
This gives a total of only $L:=1$ populated layers if $2^k-m=1$, or $L:=\sum_{j=0}^{i-1}(2^k-m)^{j}=\dfrac{(2^k-m)^{i}-1}{(2^k-m)-1}$ populated layers otherwise.
The order of the $b_\ell$'s of the populated layers is important in preventing any two copies of $P$ from being related, but as long as the order of the populated layers is maintained, the $b_\ell$'s for the populated layers may be chosen close to the middle level, i.e., $\abs{b_\ell-\floor{n/2}}\leq(k+1)L$, where $L$ is a constant defined above that does not depend on $n$ and $k+1$ is the number of levels in each layer.
Each layer has $\dbinom{n-k(j_\ell+1)}{b_\ell-\abs{R_\ell}}$ copies of $P$.
This is now asymptotic to $\dfrac{1}{(2^k)^{j_\ell+1}}\nchn$ copies since $b_\ell-\abs{R_\ell}$ is at most a fixed, finite distance from $n/2$.
This results in our desired number of copies of $P$,
$$\dfrac{\abs{\F}}{\abs P}\sim\sum_{j=0}^{i-1}\bigg(\dfrac{(2^k-m)^{j}}{(2^k)^{j+1}}\bigg)\nchn.$$
We will now demonstrate how each $R_\ell$ is chosen and in what order are the populated layers to ensure that the copies of $P$ are pairwise unrelated.

Let's start with $j=0$.
We start by populating one layer of $F$; let $\F\supseteq\{A\cup B\mid A\in f(P), B\in[n]\setminus[k],\abs B=\floor{n/2}\}$.
In other words, the layer $\ell$ with $b_\ell=\floor{n/2}$ is populated with $R_\ell=\emptyset$ and $j_\ell=0$.
Now $\abs{\F}\geq\abs{P}\dbinom{n-k}{\lfloor n/2\rfloor}$, which is asymptotically $\dfrac{1}{2^k}\nchn$ copies of $P$.
So if $m=2^k$, (i.e., $\overline{f(P)}=\B_k$,) we are done.
If not, we would like to add more copies of $P$ then just those in our middle layer so we will need to know which of the elements of $\B_n$ are available to include in the family; we consider which elements of $\B_n$ are unrelated to any element of this middle, populated layer.
Consider a set $B\in\B_n$ and $B_{[k]}:=B\cap [k]$ and $b:=\abs{B\setminus B_{[k]}}$.
This set $B$ is unrelated to all sets in $\F$ if and only if one of the following is true:
\begin{enumerate}
\item $B_{[k]}$ is unrelated to all sets in $f(P)$;
\item $B_{[k]}\nsubseteq C$ for all $C\in f(P)$ and $b<\floor{n/2}$; or
\item $B_{[k]}\nsupseteq C$ for all $C\in f(P)$ and $b>\floor{n/2}$.
\end{enumerate}
The choices for $B_{[k]}$ that can provide more sets to add to the family are exactly the sets $B_{[k]}\in \B_k\setminus \overline{f(P)}$.
In fact, each one of the $B_{[k]}\in \B_k\setminus \overline{f(P)}$ can lead to a distinct \emph{layer} of copies of $P$ by choosing the base levels correctly; the new layers are the layers from the second iteration (so would have $j_\ell=1$), and the layer's restriction set would be $B_{[k]}$.
The next step is identifying appropriate base levels for each new layer and then demonstrating how this process iterates.

Let's order the elements of $U:=\B_k\setminus \overline{f(P)}$. First, split $U$ into two sets, $U^+$ and $U^-$:
\begin{align*}
U^+:=&\B_k\setminus U(f(P))\\
=&\{V\in U\mid V\nsupseteq C\text{ for all }C\in \overline{f(P)}\},\text{ and}\\
U^-:=&U\setminus U^+\\
=&U(f(P))\setminus \overline{f(P)}\\
\subseteq&\{V\in U\mid V\nsubseteq C\text{ for all }C\in \overline{f(P)}\}.
\end{align*}
The set $U^+$ contains both the elements of $U$ contained in some element of $f(P)$ and the subsets of $[k]$ that are unrelated to any element of $f(P)$. On the other hand, $U^-$ contains the elements of $U$ containing some element of $f(P)$.
Let $\leq_U$ be any ordering of the elements in $U$ such that if $V_1\in U^-$ and $V_2\in U^+$, then $V_1\leq_U V_2$, else if $V_1\supseteq V_2$, then $V_1\leq_U V_2$.
We will use this ordering $\leq_U$ to order the base ranks to guarantee all copies of $P$ remain unrelated.

For $j=0$, we have the populated layer corresponding to $(0,\emptyset,\floor{n/2})$.
For $j=1$, we populate the layers corresponding to $(1,V,b_V)$ for each $V\in U$. We can choose the $b_V$'s such that if $V\in U^-$, then $b_V<\floor{n/2}$, and if $V\in U^+$, then $b_V>\floor{n/2}$, and if $V_1<_U V_2$, then $b_{V_1}<b_{V_2}$.
For an iteration $j>1$, for each layer corresponding to $(j-1,R,b)$ populated in iteration $j-1$, we can populate $2^k-m$ new layers, one for each set in $U$.
These new layers correspond to $\Big(j,R\cup\big(V+k(j-1)\big),b_\ell\Big)$ for each $V\in U$.
Inductively, there are then $(2^k-m)^{j}$ layers populated in iteration $j$, each with asymptotically $\dfrac{1}{(2^k)^{j+1}}\nchn$ copies of $P$, for a total of $\dfrac{(2^k-m)^{j}}{(2^k)^{j+1}}\nchn$ copies of $P$ associated with iteration $j$.
All that is left to prove is that we can put the layers in an appropriate order, i.e., the base ranks $(b_\ell)$ may be chosen in such a way as to prevent any two copies of $P$ from being related.

Let $(\ell_s)_{1\leq s\leq L}$ be the sequence of populated layers, $\ell_s$ corresponding to $(j_s,R_s,b_s)$, in the order of the rank of the base levels, i.e., for all $s_1<s_2$, $b_{s_1}<b_{s_2}$.
Let's consider our ordered set $U$ again.
Let's add to $U$ the character $E$ to indicate the `end' of a word.
Let $E$ be between $U^-$ and $U^+$ in $\leq_U$.
Consider words $V_0V_1\dots V_{j-1}E$, where the letters come from $U$, the words always end in $E$, and $E$ is only at the end of a word.
We only consider words of length $j+1$, where $0\leq j\leq i-1$.
There is a bijection between the layers $(\ell_s)$ and the possible words of length at most $i$.
Specifically, given a word $V_0V_1\dots V_{j-1}E$, its corresponding $j_s$ is $j$ and $R_s=\cup_{p=0}^{j-1}(V_p+kp)$.
Let $W$ be the set of all words of length at most $i$.
Order these words lexicographically using $\leq_U$.
Specifically, for any two words in $W$, $w_p=U_0\dots U_s$ and $w_q=V_0\dots V_t$, we say $w_1<w_2$ if and only if $U_i=V_i$ for $0\leq i\leq j-1$ and $U_j<_UV_j$ for some $j\geq0$.
Use this ordering of $W$ and the bijection between the words and the layers to directly define the corresponding ordering of the $(b_s)$.
Specifically, for two layers $\ell_1$ and $\ell_2$ with base level ranks $b_1$ and $b_2$ respectively, $b_1<b_2$ if and only if the word corresponding to $\ell_1$ is less than the word corresponding to $\ell_2$.


Now we show that no two copies of $P$ are related.
We have already seen that no two copies of $P$ in the same layer can be related.
For two copies of $P$, $P_p$ in layer $\ell_p$ (with base rank $b_p$) and $P_q$ in layer $\ell_q$ (with base rank $b_q$), consider their corresponding words, $w_p=U_0\dots U_s$ and $w_q=V_0\dots V_t$.
Without loss of generality, let $w_p<_U w_q$ so $b_p<b_q$.
Consider the subscript $c$ for the first character where $w_p$ and $w_q$ differ, i.e., $U_0\dots U_{c-1}=V_0\dots V_{c-1}$ and $U_c\neq V_c$, $U_c<_UV_c$.
Choose any representatives of the copies of $P$, $A_p\in P_p$ and $A_q\in P_q$, and define $B_p:=A_p\cap S_c$ and $B_q=A_q\cap S_c$.
Since $b_p<b_q$, we have that $A_q\nsubseteq A_p$; next we show that $A_p\nsubseteq A_q$.

The order of the words, and hence the order of the $b_\ell$'s, was chosen specifically to prevent any copies of $P$ from being pairwise related.
If $U_c=E$, then $(B_p-kc)\in f(P)$ and $V_c\in U^+$ so $V_c\nsupseteq C$ for all $C\in f(P)$, i.e., $(V_c+kc)= B_q\nsupseteq B_p$ for any $B_p$ such that $(B_p-kc)\in f(P)$.
But $B_q\nsupseteq B_p$ implies $A_q\nsupseteq A_p$.
For similar reasoning, if $V_c=E$, then $A_p\nsubseteq A_q$.
If neither $U_c=E$ nor $V_c=E$, then $U_c<V_c$ implies $U_c\nsubseteq V_c$, but $U_c=(B_p+kc)$ and $V_c=(B_q+kc)$ so $B_p\nsubseteq B_q$ so $A_p\nsubseteq A_q$.
Either way, no set from $P_p$ is related to any set from $P_q$.
This completes the proof of Theorem~\ref{main:one}.


\section{Concluding Remarks}\label{sec:REM}

We now explain how we may modify the proof above to prove Theorem~\ref{main:two}.
In proving the upper bound of Theorem~\ref{main:one}, we use the fact that the closure of a copy of $P$ meets at least $\overline a(c(P),n)$ full chains in $\B_n$.
For Theorem~\ref{main:two}, using only strong embeddings, we have that a copy of $P$ meets at least $\overline a(c^*(P),n)$, which similarly gives us the upper bound.
In the lower bound of Theorem~\ref{main:one}, we created a family $\F\subseteq\B_n$ constructed from multiple copies of $f(P)$, a weak embedding of $P$ into $\B_k$ such that $\overline{f(P)}=c(P)$.
If we instead take $f$ to be a strong embedding such that $\overline{f(P)}=c^*(P)$, then the same method of construction will achieve the asymptotic lower bound.

For a finite collection of posets, the quantities $\pa(n,\{P_1,\dots,P_k\})$ and $\pas(n,\{P_1,\dots,P_k\})$ may be found asymptotically as well. Specifically, as $n$ goes to infinity,
$$\pa(n,\{P_1,P_2,\dots,P_k\})\sim\max_{1\leq i\leq k}\Bigg(\frac{\abs{P_i}}{c(P_i)}\Bigg)\nchn,\ \text{and}$$
$$\pas(n,\{P_1,P_2,\dots,P_k\})\sim\max_{1\leq i\leq k}\Bigg(\frac{\abs{P_i}}{c^*(P_i)}\Bigg)\nchn.$$

As for future work, it would be nice to know how to find $c(P)$ and $c^*(P)$ quickly for any $P$, or at least to find the complexity of such an algorithm.
Also, in the examples in the introduction, the exact values for $\pa(n,P)$ are found, not just their asymptotic values.
It would be nice to have exact values for $\pa(n,P)$ and $\pas(n,P)$.


\section{Acknowledgment}\label{sec:ACK}

Discussions with Richard Anstee were valuable when we first formulated the packing problem and
solved it for $P=\V_2$.



\begin{thebibliography}{99}

\bibitem{AkiFra}
Jin Akiyama, P. Frankl,
{\em Modern Combinatorics (Japanese)},
Kyoritsu, Tokyo, 1987.

\bibitem{Bol} 
B. Bollob\'as, {On generalized graphs}, 
{\em Acta Math.\ Acad.\ Sci.\ Hungar.} {\bf 16}
(1965), 447--452.

\bibitem{GriLiLu} J. R. Griggs, W.-T. Li, and L. Lu,
{Diamond-free Families},
{\em J. Combinatorial Theory (Ser. A)} {\bf 119}
(2012), 310--322.

\bibitem{GriStaTro}
J. R. Griggs, J. Stahl, and W. T. Trotter, Jr.,
{A Sperner Theorem on Unrelated Chains of Subsets},
{\em J. Combinatorial Theory (Ser. A)} {\bf 36} (1984), 124--127.

\bibitem{Kat10} G. O. H. Katona, On the maximum number of incomparable
copies of a given poset in a family of subsets, Lecture,
International Conference on Recent Trends in Graph Theory and Combinatorics,
Cochin University of Science and Technology, India (2010).

\bibitem{KatNag} G. O. H. Katona and D. T. Nagy, {Independent copies of a poset
in the Boolean lattice}, draft (2013).

\bibitem{KatTar}  G. O. H. Katona and T. G. Tarj\'an,
{Extremal problems with excluded subgraphs in the $n$-cube},
in: M. Borowiecki, J. W. Kennedy, and M. M. Sys\l o (eds.) {\bf Graph Theory}, \L ag\'ow, 1981, {\em Lecture Notes in Math.}, {\bf 1018} 84--93, Springer, Berlin Heidelberg New York Tokyo, 1983.

\bibitem{KraMarYou} L. Kramer, R. R. Martin, and M. Young,
{On diamond-free subposets of the Boolean lattice},
{\em J. Combinatorial Theory (Ser. A)} {\bf 120}
(2012), 545--560.

\bibitem{Lub}
D. Lubell,
{A short proof of Sperner's lemma},
{\em J. Combin. Theory} {\bf 1} (1966), 299.

\bibitem{Spe}
E. Sperner,
{Ein Satz\"{u}ber Untermegen einer endlichen Menge},
{\em Math. Z.} {\bf 27} (1928), 544--548.

\end{thebibliography}
\end{document}